\documentclass[11pt]{amsart}
\usepackage{amsmath,amsthm, amscd, amssymb, amsfonts}
\usepackage[]{color} 

\usepackage{graphicx} 

\usepackage{mathrsfs} 
\usepackage{figsize}
\usepackage{algorithm}
\usepackage{algpseudocode}
\usepackage{enumerate} 

\usepackage{float} 

\usepackage{tikz} 

\usepackage{tkz-berge} 
\usepackage{tkz-graph}
\usetikzlibrary{arrows}
\usetikzlibrary{calc}
\usetikzlibrary{matrix}
\usetikzlibrary{decorations.pathreplacing}  

\usepgflibrary{shapes.geometric} 
\usetikzlibrary[topaths]
\usetikzlibrary{positioning}
\usetikzlibrary{arrows}

\usepackage{tabularx}

\theoremstyle{plain}
\newtheorem{theorem}{Theorem}[section]
\newtheorem{corollary}[theorem]{Corollary}
\newtheorem{proposition}[theorem]{Proposition}
\newtheorem{lemma}[theorem]{Lemma}
\newtheorem{definition}{Definition}[section]  
 
\newtheorem{remark}[theorem]{Remark}
\newtheorem{example}[theorem]{Example}

\usepackage[colorlinks]{hyperref}

\def\N{\mathbb{N}}

\def\ad{\mathrm{ad}}
\def\der{\mathrm{Der}}

\def\nil{\mathrm{nil}}

\title[The nilpotent Graph]{The Nilpotent Graph of a finite-dimensional Lie Algebra}

\thanks{...}

\author{David Towers}

 \address{Lancaster University, Department of Mathematics and Statistics, 
 }
  \email{d.towers@lancaster.ac.uk
 }

 \author{Ismael Gutierrez}
 \address{Universidad del Norte,  Departamento de Matem\'aticas y Estad\'istica, Km 5 via a Puerto Colombia, Barranquilla, Colombia.}
  \email{isgutier@uninorte.edu.co
 }

\author{Luis Fernandez}
\address{Universidad del Norte,  Departamento de Matem\'aticas y Estad\'istica, Km 5 via a Puerto Colombia, Barranquilla, Colombia.}
   \email{lfernandeze@uninorte.edu.co}

\subjclass[2010]{17B30, 17B45, 05C40, 05C25}
\keywords{Graphs associated with Lie Algebras; nilpotent Lie Algebra; hypercenter of a Lie Algebra; nilpotentizer; strongly self-centralizing Lie subalgebras}

\begin{document}

\begin{abstract}
Let $L$ be a finite-dimensional Lie algebra over a field $F$. In This paper we introduce the \emph{nilpotent graph} $\Gamma_\mathfrak{N}(L)$ as the graph whose vertices are the elements of $L \setminus \nil(L)$, where
\[\nil(L) = \{x \in L \mid \langle x, y \rangle \text{ is nilpotent for all } y \in L\},\]
and where two vertices $x, y$ are adjacent if the Lie subalgebra they generate is nilpotent. We give some characterizations of $\nil(L)$ and its connection with the hypercenter $Z^*(L)$, for example, they are equal when $F$ has characteristic zero. We prove that the nilpotentizer behaves well under direct sums, allowing a decomposition of $\Gamma_\mathfrak{N}(L)$ between components. The paper also investigates the structural and combinatorial properties of $\Gamma_\mathfrak{N}(L)$, including the conditions under which the graph is connected. We characterize the existence of strongly self-centralizing subalgebras in relation to connectivity and vertex isolation. Explicit computations are carried out for the algebra $\mathfrak{t}(2,\mathbb{F}_q)$, where $\Gamma_\mathfrak{N}(L)$ decomposes into $q+1$ components, each of size $q(q-1)$, forming a $(q^2-q-1)$-regular graph.
We conclude with algorithms for constructing $\Gamma_\mathfrak{N}(L)$ in SageMath, and pose open problems concerning bipartiteness, regularity, and structural implications in higher dimensions over finite fields.
\end{abstract}

\maketitle


\section{Introduction}
 
Nilpotent Lie algebras serve as a unifying framework in mathematics, connecting diverse areas such as representation theory, differential geometry, harmonic analysis, geometric quantization, and algebraic topology. Their study reveals deep connections between different mathematical disciplines and provides powerful tools for investigating the geometric and algebraic structures underlying these areas.

The interaction between algebraic structures and graph theory has been a fruitful study area, providing a robust framework for visualizing and analyzing algebraic objects. Several graph constructions associated with groups and rings have been introduced in recent years, yielding new insights into their structure and properties. Motivated by this, we investigate the \emph{nilpotent graph} of a finite-dimensional Lie algebra.

Given a Lie algebra $L$ over a field $F$, we define its nilpotent graph, denoted $\Gamma_\mathfrak{N}(L)$, as the graph whose vertex set is $L \setminus \nil(L)$, where $\nil(L)$ is the set of elements $x \in L$ such that $\langle x,y \rangle$ is nilpotent for every $y \in L$. Two vertices $x$ and $y$ are adjacent in $\Gamma_\mathfrak{N}(L)$ if the subalgebra they generate is nilpotent. This construction parallels related work on non-commuting graphs of groups, nilpotent graphs of groups and zero-divisor graphs of rings and captures the behavior of nilpotent substructures in a Lie-theoretic context, \cite{Gutierrez}.

In this paper, we explore the fundamental properties of $\Gamma_\mathfrak{N}(L)$, focusing on how its structure reflects that of the underlying Lie algebra. We show that $\nil(L)$ contains the hypercenter $Z^*(L)$ and that the equality holds in characteristic zero. We also demonstrate that $\Gamma_\mathfrak{N}(L)$ respects direct sums and describe the nilpotent graph of the algebra of $2 \times 2$ upper triangular matrices over a finite field. Further, we characterize conditions for connectedness, vertex isolation, regularity and study the existence of strongly self-centralizing Lie subalgebras. Finally, we provide SageMath implementations for constructing $\Gamma_\mathfrak{N}(L)$ and pose open questions for future research.

\section{Preliminaries}

\subsection*{Basics on Graph Theory.}
We recall some basic notions from graph theory that will be used throughout the paper. A \emph{graph} $\Gamma$ is an ordered pair $(V,E)$, where $V$ is a set of vertices and $E$ is a set of unordered pairs of distinct elements of $V$, called edges. If $\{u,v\} \in E$, we say that $u$ and $v$ are \emph{adjacent}, or \emph{neighbors}.

A graph is said to be \emph{simple} if it has no loops and no multiple edges, which will always be assumed in this work. The \emph{degree} of a vertex $v$, denoted $\deg(v)$, is the number of vertices adjacent to $v$. A graph is called \emph{$k$-regular} if all of its vertices have degree $k$, that is, if $\deg(v) = k$ for every vertex $v$ in the graph.

A subset $C \subseteq V$ is a \emph{clique} if every pair of distinct vertices in $C$ are adjacent. A graph is \emph{complete} if every pair of distinct vertices is adjacent; the complete graph on the $n$ vertices is denoted $K_n$. A graph is \emph{connected} if there is a path between any two vertices. A \emph{connected component} of a graph is a maximal connected subgraph.
Some useful notations are the following: $\kappa(\Gamma)$ denotes the number of connected components of $\Gamma$, $\mathcal{C}(\Gamma) :=\{C_1,\ldots, C_{\kappa(\Gamma)}\}$ is the set of all connected components of $\Gamma$, and finally, $\kappa_j(\Gamma) := |C_j|$, \ for $j=1,\ldots \kappa(\Gamma)$.

We denote by $\Gamma^c$ the \emph{complement} of a graph $\Gamma$, defined on the same set of vertex $V$ with edges $\{u,v\}$ such that $u$ and $v$ are not adjacent in $\Gamma$. The study of $\Gamma_\mathfrak{N}(L)$ and its complement $\Gamma_\mathfrak{N}(L)^c$ provides a rich framework to explore the combinatorial features of Lie algebraic nilpotency. The reader is referred for undefined terms and concepts of graph theory to \cite{Diestel} or \cite{Nora}.

\subsection*{Basics on Lie Algebras.}
Let $L$ be a finite-dimensional Lie algebra over a field $F$. A \emph{Lie algebra} is a vector space $L$ over $F$, equipped with a bilinear operation $[\cdot,\cdot]: L \times L \to L$, called the \emph{Lie bracket}, satisfying the following axioms:
\begin{enumerate}
    \item \textbf{Alternating property:} $[x,x] = 0$ for all $x \in L$;
    \item \textbf{Jacobi identity:} $[x,[y,z]] + [y,[z,x]] + [z,[x,y]] = 0$ for all $x,y,z \in L$.
\end{enumerate}
These imply \textbf{anticommutativity}: $[x,y] = -[y,x]$ for all $x,y \in L$.

A Lie algebra is said to be \emph{abelian} if $[x,y] = 0$ for all $x,y \in L$. Every one-dimensional Lie algebra is abelian. The \emph{dimension} of a Lie algebra is its dimension as a vector space over $F$.

Standard examples include:
\begin{itemize}
    \item $\mathfrak{gl}(n,F)$, the Lie algebra of all $n \times n$ matrices over $F$ with bracket $[x,y] = xy - yx$;
    \item $\mathfrak{t}(n,F)$, the subalgebra of upper triangular matrices;
    \item $\mathfrak{u}(n,F)$, the subalgebra of strictly upper triangular matrices.
\end{itemize}

A \emph{subalgebra} of $L$ is a subspace closed under the Lie bracket. If $A, B$ are subspaces of $L$, then $[A,B]$ denotes the subspace spanned by all $[a,b]$ with $a \in A$, $b \in B$. The following properties hold:
\begin{enumerate}
    \item $[A,B] = [B,A]$;
    \item $[A,B+C] = [A,B] + [A,C]$;
    \item $[A,[B,C]] \subseteq [[A,C],B] + [[B,C],A]$;
    \item $[A \cap B, C] \subseteq [A,C] \cap [B,C]$.
\end{enumerate}

A subspace $I$ of $L$ is called an \emph{ideal} if $[x,y] \in I$ for all $x \in I$, $y \in L$, that is, $[I,L] \subseteq I$. Intersections, sums, and Lie products of ideals are again ideals.

The \emph{derived series} of $L$ is defined recursively by
\begin{align*}
    L^{(0)} &= L, \\
    L^{(k+1)} &= [L^{(k)}, L^{(k)}].
\end{align*}
The \emph{lower central series} is given by
\begin{align*}
    L^1 &= L, \\
    L^{k+1} &= [L^k, L].
\end{align*}
A Lie algebra $L$ is called \emph{solvable} if $L^{(r)} = 0$ for some $r$, and \emph{nilpotent} if $L^k = 0$ for some $k$. Every nilpotent Lie algebra is solvable, but the converse is not true.

Examples:
\begin{itemize}
    \item $\mathfrak{u}(n,F)$ is nilpotent;
    \item $\mathfrak{t}(n,F)$ is solvable but not nilpotent;
    \item the Lie algebra with basis $\{x,y\}$ and bracket $[x,y] = y$ is solvable but not nilpotent. In fact, $[L',L'] = [F y, F y] =0$, and $Z(L)=0$.
\end{itemize}

For $x \in L$, define the \emph{adjoint map} $\ad\, x : L \to L$ by $\ad\, x(y) = [y,x]$. The map $\ad: L\longrightarrow \der(L)$, where $\der(L)$ is the set of derivations of $L$, is called the \emph{adjoint representation} of $L$. An ideal $I$ of $L$ is called \emph{characteristic} if it is invariant under all derivations of $L$.

The \emph{center} of $L$ is $Z(L) = \{x \in L \mid [x,L] = 0\}$. Define the upper central series recursively as
\begin{align*}
    Z_0(L) &= 0, \\
    Z_i(L)/Z_{i-1}(L) &= Z(L/Z_{i-1}(L)), \quad i \geq 1.
\end{align*}
The union $\bigcup_{i \geq 0} Z_i(L)$ is called the \emph{hypercentre}, denoted $Z^*(L)$. For finite-dimensional $L$, this series terminates.

Further definitions:
\begin{itemize}
    \item $L$ is \emph{simple} if it is nonabelian and has no nontrivial proper ideals;
    \item $L$ is \emph{semisimple} if it is a direct sum of simple Lie algebras;
    \item $L$ is \emph{almost algebraic} (char 0) if it contains the semisimple and nilpotent parts of all its elements;
    \item The \emph{nilradical} $N(L)$ is the largest nilpotent ideal of $L$;
    \item The \emph{radical} $R(L)$ is the largest solvable ideal of $L$;
    \item $L$ is \emph{completely solvable} if $L^2$ is nilpotent;
    \item $L$ is \emph{reductive} if its adjoint representation is completely reducible.
\end{itemize}

\begin{lemma}\label{z} 
Let $A$ be an ideal of the nilpotent Lie algebra $L$. Then $A\cap Z(L)\neq 0$.
\end{lemma}

\begin{proof}  
Since $L$ is nilpotent, there exists $k\geq 1$ such that $L^{k+1}\cap A=0$ but $L^k\cap A\neq 0$. Then $0\neq L^k\cap A\subseteq A\cap Z(L)$.  
\end{proof}

The radical of $L$, denoted $R(L)$, is the largest solvable ideal of $L$.

Let $x\in L$. Put $$E_L(x)=\{y\in L \mid (\ad\, x)^n(y)=0 \hbox{ for some } n\in \N\}.$$ Then $E_L(x)$ is a subalgebra of $L$ (see \cite{barnes}). We say that $y$ is an\emph{Engel element} of $L$ if $y\in E_L(x)$ for all $x\in L$. Note that this differs from the terminology used by some authors, but is aligned with the use in group theory. Let $E(L)$ be the set of all Engel elements of $L$. 

Clearly, $E(L)$ is always a subalgebra of $L$. By Engel's theorem it is nilpotent. We also have

\begin{theorem} 
Over a field of characteristic zero, $E(L)$ is a characteristic ideal of $L$, and so $E(L)\subseteq N(L)$, the nilradical of $L$.
\end{theorem}

\begin{proof} 
Let $\theta$ be an automorphism of $L$. If $y\in E_L(x)$, then $\theta(y)\in E_L(\theta(x))$, so $E(L)$ is invariant under all automorhisms of $L$. It follows from \cite[Corollary 3.2]{frat} that it is invariant under all derivations of $L$ and hence is a characteristic ideal of $L$.
\end{proof}

\section{The nilpotentizer of a Lie algebra}
 
In this paper, $\mathfrak{N}$ and $\mathfrak{N}_c$ denote the class of all finite-dimensional nilpotent Lie algebras and the class of all finite-dimensional nilpotent Lie algebras of the nilpotency class $c$, respectively.

\begin{definition}
Let $L$ be a finite-dimensional Lie algebra over a field $F$. 
\begin{enumerate}
    \item For $h\in L$ we set 
\begin{equation*}
\nil_L(h) := \{x\in L\mid  \langle h,x\rangle \ \text{is a nilpotent subalgebra of} \  L \}.
\end{equation*}
This set is called the \emph{nilpotentiser} of $h$ in $L$.

\item We define the nilpotentizer of $L$ as follows:
\begin{equation*}
\nil(L) := \{x\in L\mid  \langle h,x\rangle \ \text{is nilpotent for all} \ h\in L\}.
\end{equation*}
\end{enumerate}
\end{definition}


\subsection*{Properties of nilpotentisers}
\bigskip

If $x\in L$, then $\nil_L(x)$ is the union of all maximal nilpotent subalgebras containing $x$.

\begin{lemma} 
For any Lie algebra $L$ and any $x\in L$, the following are equivalent:
\begin{itemize}
\item[(i)]  $\nil_L(x)$ is a subspace of $L$;
\item[(ii)]  $ \nil_L(x)$  is a maximal nilpotent subalgebra of $L$;
\item[(iii)]  $x$ is contained in just one maximal nilpotent subalgebra.
\end{itemize}
\end{lemma}

\begin{proof} 
Simply note that if $U$, $V$ are subspaces of $L$, then $U\cup V$ is a subspace if and only if $U\subseteq V$ or $V\subseteq U$.
\end{proof}

\begin{lemma} Let $L$ be a Lie algebra, $J$ an ideal of $L$ and $x,y\in L$. Then
\begin{itemize}
\item[(i)] $\nil(L)\subseteq \nil_L(x)$;
\item[(ii)] $\frac{\nil_L(x)+J}{J}\subseteq \nil_L(x+J)$;
\item[(iii)] $\nil_{L/J}(x+J)=\frac{\nil_L(x)}{J}$ whenever $J\subseteq Z^*(L)$;
\item[(iv)] the nilpotentiser of every element of $L$ is a subalgebra if and only if the same is true for $L/J$ for some ideal $J$ of $L$ with $J\subseteq Z^*(L)$;
\item[(v)] if $\nil_L(x)$ is a subalgebra of $L$ and $y\in \nil_L(x)$, then $\nil_L(y)=\nil_L(x)$;
\item[(vi)] if $\theta$ is an automorphism of $L$, then $\theta(\nil_L(x))=\nil_L(\theta(x))$.
\end{itemize}
\end{lemma}

\begin{proof} 
These are straightforward.
\end{proof}

\begin{lemma} 
For every Lie algebra $L$, $\nil(L)\subseteq E(L)$.
\end{lemma}

\begin{proof} 
Let $y \in \nil(L)$ and $x\in L$. Then $\langle x,y\rangle$ is nilpotent, and so $y\in E_L(x)$.
\end{proof}

\begin{theorem}\label{z*} 
Let $L$ be a $n$-dimensional  Lie algebra over a field $F$. Then the following holds:
\begin{itemize}
\item[(i)] $\langle x, Z^*(L)\rangle$ is nilpotent for all $x\in L$; and
\item[(ii)] if $F$ has characteristic zero, then $Z^*(L)= \nil(L)$.
\end{itemize}
\end{theorem}

\begin{proof} 
(i) Let $x\in L$ and put $S=\langle x, Z^*(L)\rangle$. Then $Z^*(L)\subseteq Z^*(S)\subseteq S$, so $S=Z^*(S)+Fx$ and $S$ is nilpotent.
\medskip

\noindent (ii) It follows from (i) that $Z^*(L)\subseteq \nil(L)$. So, it suffices to show that there is no non-trivial element in $\nil(L/Z^*(L))$; in other words we may assume that $Z(L)=0$ and we must show that $\nil(L)=0$. Let $0\neq x\in \nil(L)$. Let $\theta$ be a faithful matrix representation of $L$, which exists by Ado's Theorem and let $L^*$ be the algebraic hull of $\theta(L)$ (see \cite{chev} for existence and properties of the algebraic hull) and identify $L$ with $\theta(L)$..  Let $C^*$ be a Cartan subalgebra of $L^*$. Then $C^*=E_{L^*}(u)$, where $E_{L^*}(u)$ is minimal in the set $\{E_{L^*}(y) : y\in L^*\}$, by \cite[Theorem 1]{barnes}.
\par

Now the algebraic hull is almost algebraic (see \cite{ab}), so let $u=u_s+u_n$ be the Jordan decomposition of $u$. Then $E_{L^*}(u_s)\subseteq E_{L^*}(u)$, so we have that $u$ is semisimple and $E_{L^*}(u)=C_{L^*}(u)$.  Let $c\in C^*$. Then $[u,[x,c]]=[x,[u,c]]=0$, so $[x,c]\in C_L^*(u)=C^*$. Hence $x\in N_{L^*}(C^*)=C^*$ and $[x,u]=0$. Write $L^*=C^*\dot{+} L^*_1$ where $L^*_1$ is the Fitting one-component of ad\,$u$, by \cite{barnes}. 
\par

By passing to the algebraic closure we can suppose that $F$ is algebraically closed. Let $y_i,\ldots,y_r$ be a basis for $L^*_1$ and suppose that $[y_i,u]=\alpha_iy_i$ where $\alpha_i\in F$. Note that $L^*_1=[u,L^*_1]\subseteq [L^*,L^*]=[L,L]\subseteq L$ (see \cite[Proposition 1]{chev}. So $\langle x,y_i\rangle$ is nilpotent and there is a $k$ such that $(\ad\, y_i)^{k+1}(x)=0$, but $(\ad\, y_i)^k(x)\neq 0$. Put $S^*=\langle x,u+y_i\rangle$. Then $[[x,y_i],u]=[[y_i,u],x]=\alpha_i[x,y_i]$ and a straightforward induction proof shows that $[(\ad\, y_i)^n(x),u]=n\alpha_i(\ad\, y_i)^n(x)$. Now $[[[x,y_i],u+y_i]=\alpha_i[x,y_i]+[[x,y_i]y_i]$, whence $[[x,y_i]y_i]\in S^*$. Similarly, $(\ad\, y_i)^t(x)\in S^*$ for all $t$.  Now
\[ (\ad\, (u+y_i))^n((\ad\, y_i)^k(x))=(\ad\, u)^n(\ad\, y_i)^k(x)),
\]
so $(\ad\, y_i)^k(x)\in C^*$. It follows that 
\[ 0=[(\ad\, y_i)^k(x),u]=k\alpha_i(\ad\, y_i)^k(x).
\]
But, $\alpha_i=0$ implies that $y_i\in C^*\cap L^*_1=0$, so $k=0$ and $[y_i,x]=0$. Hence $x\in Z_{C^*}(L^*_1)\cap L$ and $ Z_{C^*}(L^*_1)\cap L$ is a non-trivial ideal of $C^*$. It must, therefore have non-trivial intersection with $Z(C^*)$, by Lemma \ref{z}, whence $Z(L^*)\cap L\neq 0$. But $Z(L^*)\cap L\neq 0\subseteq Z(L)$, contradiction.
\end{proof}


\begin{lemma} let $L$ be a completely solvable Lie algebra over any field $F$. Then $\nil(L)\subseteq N(L)$, the nilradical of $L$.
\end{lemma}
\begin{proof} Let $x\in \nil(L)$. Then $Fx+N(L)$ is nilpotent. Moreover, it is an ideal of $L$, since $L^2\subseteq N(L)$. Hence $x\in N(L)$.   
\end{proof} 

\begin{lemma} Let $L$ be a classical simple Lie algebra over an infinite field of characteristic $p > 3$, or a semisimple
 Lie algebra in characteristic 0. Then $\nil(L)=0$.
\end{lemma}

\begin{proof} 
This follows from \cite[Theorem B and section 1.2.2]{bois}.
\end{proof}




\begin{remark}
Let $L$ be a reductive Lie algebra over a field $F$ of characteristic zero. Then $L$ decomposes as:
\[L = S \oplus Z(L),\]
where $S = [L, L]$ is semisimple and $Z(L)$ is the center of $L$.
\end{remark}

\begin{lemma}
Let $L$ be a reductive Lie algebra over a field $F$ of characteristic zero. Then $\nil(L) = Z(L)$.
\end{lemma}

\begin{proof} 
Let $x \in Z(L)$. Then for any $y \in L$, we have $[x, y] = 0$, so $\langle x, y\rangle$ is abelian, hence nilpotent. Thus, $x \in \nil(L)$.

Reciprocally, let $x \in \nil(L)$. Write $x = s + z$ with $s \in S$, and $z \in Z(L)$. Suppose $s \neq 0$.

Since $S$ is semisimple and the field has characteristic zero, there exists $y \in S \subseteq L$ such that the subalgebra $\langle s, y \rangle$ is not nilpotent. But $z \in Z(L)$, so for all $y \in L$, we have:
\[[x, y] = [s + z, y] = [s, y] + [z, y] = [s, y].\]
Hence \( \langle x, y \rangle = \langle s, y \rangle \), which is not nilpotent, a contradiction. Therefore, $s = 0$, and $x = z \in Z(L)$. Since $S$ is semisimple and the field has characteristic zero, there exists \( y \in S \subseteq L \) such that the subalgebra \( \langle s, y \rangle \) is not nilpotent.
\end{proof}

\section{The nilpotent graph of a Lie algebra}

\begin{definition}
Let $L$ be a finite-dimensional non-nilpotent Lie algebra. The nilpotent graph of $L$, denoted by $\Gamma_{\mathfrak{N}}(L)$, is a simple undirected graph whose vertex set is $L\setminus \nil(L)$, and two vertices $x$ and $y$ are adjacent if and only if $\langle x, y\rangle$ is a nilpotent subalgebra of $L$.    
\end{definition}

\begin{example}
The nilpotent graph of $\mathfrak{t}(2,\mathbb{F}_2)$ consists of three connected components and each one of them is $K_{2}$
    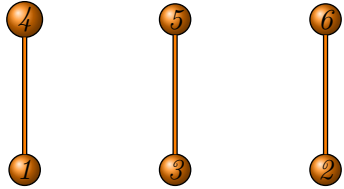
\begin{figure}[H]
    \centering
    \begin{tikzpicture}
        \GraphInit[vstyle=Shade]
        \tikzset{VertexStyle/.append style={minimum size=1pt, inner sep=1pt}}
        \renewcommand*{\EdgeLineWidth}{0.5pt}
        \Vertex{1}
        \Vertex[x=4 , y=0, ]{2}    
        \Vertex[x=4 , y=2]{6}  
        \Vertex[x=2 , y=0]{3} 
        \Vertex[x=0 , y=2]{4}   
        \Vertex[x=2 , y=2]{5}  
        \Edges(3,5)
        \Edges(1,4)
        \Edges(2,6)
    \end{tikzpicture}
    \caption{Nilpotent Graph of $\mathfrak{t}(2,\mathbb{F}_2)$}
    \end{figure}
    
    \begin{table}[H]
    \centering
    \begin{tabular}{||c|c||c|c||} 
        \hline
         Index & Matrix & Index & Matrix \\ [0.5ex] 
         \hline\hline
         1 & $\begin{pmatrix}
             0&0\\ 0 &1\\
         \end{pmatrix}$ &  4 & $\begin{pmatrix}
             1&0\\ 0 &0\\
         \end{pmatrix}$\\ 
         \hline
         2 & $\begin{pmatrix}
             0&1\\ 0 &0\\
         \end{pmatrix}$ & 5 & $\begin{pmatrix}
             1&1\\ 0 &0\\
         \end{pmatrix}$ \\
         \hline
         3 & $\begin{pmatrix}
             0&1\\ 0 &1\\
         \end{pmatrix}$ & 6 & $\begin{pmatrix}
             1&1\\ 0 &1\\
         \end{pmatrix}$ \\ [1ex] 
         \hline
    \end{tabular}
    \caption{Index and matrix representation}
    \label{tab:index_matrix}
    \end{table}

    In $\mathfrak{t}(2,\mathbb{F}_2)$ there are only two invertible matrices:
    $$ id = \begin{pmatrix}
             1&0\\ 0 &1\\
         \end{pmatrix} \quad A = \begin{pmatrix}
             1&1\\ 0 &1\\
         \end{pmatrix} $$
\end{example}

\begin{example}
The nilpotent graph of $\mathfrak{t}(2,\mathbb{F}_3)$ consists of four connected components, and each one of them is $K_{6}$
    \begin{figure}[H]
\begin{center}
    \begin{tikzpicture}
        \GraphInit[vstyle=Shade]
        \tikzset{VertexStyle/.append style={minimum size=1pt, inner sep=1pt}}
        \renewcommand*{\EdgeLineWidth}{0.5pt}
        \Vertex[x=-8.5,y=1]{1}
        \Vertex[x=-8 , y=0]{2}    
        \Vertex[x=-7 , y=0]{9} 
        \Vertex[x=-6.5 , y=1]{10}   
        \Vertex[x=-7 , y=2]{17}  
        \Vertex[x=-8 , y=2]{18}  
        \Edges(1,2)
        \Edges(1,9)
        \Edges(1,10)
        \Edges(1,17)
        \Edges(1,18)
        \Edges(2,9)
        \Edges(2,10)
        \Edges(2,17)
        \Edges(2,18)
        \Edges(9,10)
        \Edges(9,17)
        \Edges(9,18)
        \Edges(10,17)
        \Edges(10,18)
        \Edges(17,18)        
\end{tikzpicture}
\hskip0.5cm
\begin{tikzpicture}
        \GraphInit[vstyle=Shade]
        \tikzset{VertexStyle/.append style={minimum size=1pt, inner sep=1pt}}
        \renewcommand*{\EdgeLineWidth}{0.5pt}
        \Vertex[x=-4.5,y=1]{5}
        \Vertex[x=-4 , y=0]{7}    
        \Vertex[x=-3 , y=0]{11} 
        \Vertex[x=-2.5 , y=1]{16}   
        \Vertex[x=-3 , y=2]{20}  
        \Vertex[x=-4, y=2]{22}
        \Edges(5,7)
        \Edges(5,11)
        \Edges(5,16)
        \Edges(5,20)
        \Edges(5,22)
        \Edges(7,11)
        \Edges(7,16)
        \Edges(7,20)
        \Edges(7,22)
        \Edges(11,16)
        \Edges(11,20)
        \Edges(11,22)
        \Edges(16,20)
        \Edges(16,22)
        \Edges(20,22)       
\end{tikzpicture}
\hskip0.5cm
\begin{tikzpicture}
        \GraphInit[vstyle=Shade]
        \tikzset{VertexStyle/.append style={minimum size=1pt, inner sep=1pt}}
        \renewcommand*{\EdgeLineWidth}{0.5pt}
        \Vertex[x=-0.5,y=1]{4}
        \Vertex[x=0 , y=0]{8}    
        \Vertex[x=1 , y=0]{13} 
        \Vertex[x=1.5 , y=1]{14}   
        \Vertex[x=1 , y=2]{19}  
        \Vertex[x=0, y=2]{23} 
        \Edges(4,8)
        \Edges(4,13)
        \Edges(4,14)
        \Edges(4,19)
        \Edges(4,23)
        \Edges(8,13)
        \Edges(8,14)
        \Edges(8,19)
        \Edges(8,23)
        \Edges(13,14)
        \Edges(13,19)
        \Edges(13,23)
        \Edges(14,19)
        \Edges(14,23)
        \Edges(19,23)
\end{tikzpicture}
\hskip0.5cm
\begin{tikzpicture}
        \GraphInit[vstyle=Shade]
        \tikzset{VertexStyle/.append style={minimum size=1pt, inner sep=1pt}}
        \renewcommand*{\EdgeLineWidth}{0.5pt}
        \Vertex[x=-4.5,y=-2]{3}
        \Vertex[x=-4 , y=-3]{6}    
        \Vertex[x=-3 , y=-3]{12} 
        \Vertex[x=-2.5 , y=-2]{15}   
        \Vertex[x=-4 , y=-1]{21}
        \Vertex[x=-3 , y=-1]{24}
        \Edges(3,6)
        \Edges(3,12)
        \Edges(3,15)
        \Edges(3,21)
        \Edges(3,24)
        \Edges(6,12)
        \Edges(6,15)
        \Edges(6,21)
        \Edges(6,24)
        \Edges(12,15)
        \Edges(12,21)
        \Edges(12,24)
        \Edges(15,21)
        \Edges(15,24)
        \Edges(21,24)
    \end{tikzpicture}
\end{center}
    \caption{Nilpotent Graph of $\mathfrak{t}(2,\mathbb{F}_3)$}
    \end{figure}
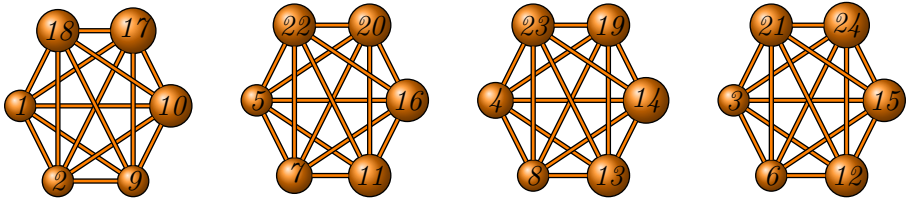

    \begin{table}[H]
    \centering
    \begin{tabular}{||c|c||c|c||c|c||c|c||} 
        \hline
         Index & Matrix & Index & Matrix & Index & Matrix & Index & Matrix \\ [0.5ex] 
         \hline\hline
         1 & $\begin{pmatrix}
             0&0\\ 0 &1\\
         \end{pmatrix}$ & 7 & $\begin{pmatrix}
             0&2\\ 0 &1\\
         \end{pmatrix}$ & 13 & $\begin{pmatrix}
             1&1\\ 0 &2\\
         \end{pmatrix}$ & 19 & $\begin{pmatrix}
             2&1\\ 0 &0\\
         \end{pmatrix}$ \\ 
         \hline
         2 & $\begin{pmatrix}
             0&0\\ 0 &2\\
         \end{pmatrix}$ & 8 & $\begin{pmatrix}
             0&2\\ 0 &2\\
         \end{pmatrix}$ & 14 & $\begin{pmatrix}
             1&2\\ 0 &0\\
         \end{pmatrix}$ & 20 & $\begin{pmatrix}
             2&1\\ 0 &1\\
         \end{pmatrix}$ \\
         \hline
         3 & $\begin{pmatrix}
             0&1\\ 0 &0\\
         \end{pmatrix}$ & 9 & $\begin{pmatrix}
             1&0\\ 0 &0\\
         \end{pmatrix}$ & 15 & $\begin{pmatrix}
             1&2\\ 0 &1\\
         \end{pmatrix}$  & 21 & $\begin{pmatrix}
            2&1\\ 0 &2\\
         \end{pmatrix}$ \\
         \hline
         4 & $\begin{pmatrix}
             0&1\\ 0 &1\\
         \end{pmatrix}$ & 10 & $\begin{pmatrix}
             1&0\\ 0 &2\\
         \end{pmatrix}$ & 16 & $\begin{pmatrix}
             1&2\\ 0 &2\\
         \end{pmatrix}$ & 22 & $\begin{pmatrix}
             2&2\\ 0 &0\\
         \end{pmatrix}$ \\
         \hline
         5 & $\begin{pmatrix}
             0&1\\ 0 &2\\
         \end{pmatrix}$ & 11 & $\begin{pmatrix}
             1&1\\ 0 &0\\
         \end{pmatrix}$ & 17 & $\begin{pmatrix}
             2&0\\ 0 &0\\
         \end{pmatrix}$ & 23 & $\begin{pmatrix}
             2&2\\ 0 &1\\
         \end{pmatrix}$ \\ 
         \hline
         6 & $\begin{pmatrix}
             0&2\\ 0 &0\\
         \end{pmatrix}$ & 12 & $\begin{pmatrix}
             1&1\\ 0 &1\\
         \end{pmatrix}$ & 18 & $\begin{pmatrix}
             2&0\\ 0 &1\\
         \end{pmatrix}$ & 24 & $\begin{pmatrix}
             2&2\\ 0 &2\\
         \end{pmatrix}$  \\
         \hline
    \end{tabular}
    \caption{Index and matrix representation}
    \label{tab:index_matrix2}
    \end{table}

    In $\mathfrak{t}(2,\mathbb{F}_3)$, there exist 12 invertible matrices. The only one that does not live in one of the components is 
    $$A = \begin{pmatrix}
        2&0\\
        0&2
    \end{pmatrix}$$
\end{example}

\begin{remark}
\begin{enumerate}    
    \item The nilpotent graph of $\mathfrak{t}(2,\mathbb{F}_4)$ consists of five connected components, and each one of them is $K_{12}$

    \item The nilpotent graph of $\mathfrak{t}(2,\mathbb{F}_5)$ consists of six connected components, and each one of them is $K_{20}$
\end{enumerate}
\end{remark}







 This suggests the following question.   
\medskip

\textbf{Question.}
Does the nilpotent graph of $\mathfrak{t}(2,\mathbb{F}_q)$ consist of $q+1$ connected components, and each one of them is $K_{q(q-1)}$?
\bigskip

\noindent \underline{The algebra $L=\mathfrak{t}(2,F_n)$}
\medskip

Put $z=\begin{pmatrix} 1&0\\0&1\end{pmatrix}$, $y=\begin{pmatrix}0&1\\0&0\end{pmatrix}$, $F=F_n$. Then $\nil(L)=Z(L)=Fz$, $L^2=Fy$.

\begin{lemma} 
If $x\notin Z(L)$, then $\nil_L(x)$ is a two-dimensional abelian subalgebra of $L$.
\end{lemma}

\begin{proof} 
Certainly, $\langle x,z\rangle$ is a two-dimensional subalgebra which is contained in $\nil_L(x)\subseteq E_L(x)$. Now $E_L(x)$ is a subalgebra of $L$. If $E_L(x)$ is two-dimensional, then $\nil_L(x)=Fx+Fz$. 
\par

So suppose $E_L(x)=L$. Let $w\in L\setminus \langle x,z\rangle$, so $L=Fz+Fx+Fw$. Then $[x,w]=\lambda y$  and $[x,y]=\mu y$ for some $\lambda, \mu\in F$. Since $w\in E_L(x)$, $\lambda\mu=0$. If $\lambda=0$, then $L$ is abelian, a contradiction, so $\mu=0$. Let $[w,y]=\nu y$. Then $\nu\neq 0$, since, otherwise, $L$ is nilpotent. Hence $\nil_L(x)=Fx+Fz$.
\end{proof}

\begin{theorem}
The nilpotent graph of $\mathfrak{t}(2,\mathbb{F}_q)$ consist of $q+1$ connected components, and each one of them is $K_{q(q-1)}$.
\end{theorem}
\begin{proof}
Now $\nil(L)$ is the set of scalar matrices with elements in $F_q$ and there are $q$ of these. Since $L$ has $q^3$ elements, the number of vertices in the graph is $ q^3-q$. If $x\notin Z(L)$, $\nil_L(x)$ is two-dimensional, and so has $q^2$ elements. Thus, the component of the graph containing $x$ has $q^2-q$ vertices; as $\nil_L(x)$ is abelian, that component will be $K_{q(q-1)}$. If $\nil_L(x)\neq \nil_L(y)$ we have that $L=\nil_L(x)+\nil_L(y)$, so $\dim (\nil_L(x)\cap \nil_L(y)=1$, whence  $\nil_L(x)\cap \nil_L(y)=\nil(L)$. Hence, the total number of components in the graph is $(q^3 - q)/(q^2 - q) = q + 1$.
\end{proof}

As a consequence, we have the following.

 \begin{corollary}
The nilpotent graph of $\mathfrak{t}(2,\mathbb{F}_q)$ is $(q^2-q-1)$-regular.     
 \end{corollary}

 We have the following results with the same argument as in Lemma 2.1 and Proposition 2.2 in \cite{Bhowal}.


\begin{lemma}\label{grado}
Let $L$ be a finite-dimensional Lie algebra over $F_q$. 
\begin{enumerate}[(i)]
    \item Then for every $h\in L\setminus \nil(L)$ we have
\begin{equation}\label{grado2}
  \deg(h) = |\nil_L(h)| - |\nil(L)| - 1.  
\end{equation}
    \item The nilpotent graph of $L$ is never a star graph.
\end{enumerate}
\end{lemma}

\begin{proof}
(i) Since $h\in \nil_L(h)$, we see that $|\nil_L(h)| - 1$ is exactly the number of vertices that are adjacent to $h$. Due to the vertex set being $L\setminus \nil(L)$, it follows that $ \deg(h) = |\nil_L(h)| - |\nil(L)| - 1$. 

(ii) Suppose that the nilpotent graph of $L$ is a star. Then there exists $x\in L\setminus \nil(L)$ such that $\deg(x) = |L- \nil(G)| - 1 =  |L|-|\nil(G)|-1$. Thus,  $|\nil_L(x)| = |L|$ and so $x\in \nil(L)$, which is a contradiction.
\end{proof}

\begin{corollary}
Let $L$ be a finite-dimensional solvable Lie algebra over $F_q$. If $\dim(\nil(L))$ is even and divides $\dim (\nil_L(x))$ for all $x\in L$, then $\Gamma_\mathfrak{N}(L)$ is not Eulerian.
\end{corollary}

\begin{proof}
It follows straightforward from \eqref{grado2}.
\end{proof}

\begin{lemma}
Let $U$ be a nilpotent subalgebra of a finite-dimensional Lie algebra $L$, such that $U\not\subseteq \nil(G)$. Then the set 
$U\setminus (U\cap \nil(L))$ forms a clique in $\Gamma_\mathfrak{N}(L)$.
\end{lemma}

\begin{proof}
For all $x,y\in U\setminus (U\cap \nil(L))$ we have that $\langle x,y\rangle$ is nilpotent. Thus, the subgraph of $\Gamma_\mathfrak{N}(L)$  having $U- (U\cap \nil(L))$ as its set of vertices is a clique. 
\end{proof}

\section{Direct Sums}

\begin{lemma}\label{ComponentsNilAdjacency}
Let $L_1$ and $L_2$ be Lie algebras over a field $F$, $a_1, a_2\in L_1$ and $b_1, b_2\in L_2$, Then $\langle (a_1, b_1), (a_2, b_2) \rangle$ is a nilpotent Lie subalgebra of $L_1\oplus L_2$ if and only if $\langle a_1, a_2 \rangle$ and $\langle b_1, b_2 \rangle$ are nilpotent Lie subalgebras of $L_1$ and $L_2$ respectively.
\end{lemma}

\begin{proof}
It is clear that if $(x,y) \in \langle(a,b),(c,d)\rangle$ then $x \in \langle a,c\rangle$ and $y \in \langle b,d\rangle$ (a simple induction would give a formal proof), so $\langle(a,b),(c,d)\rangle \subseteq \langle a,c\rangle \oplus \langle b,d\rangle$. Hence $\langle a,c\rangle$ and $\langle b,d\rangle$ nilpotent implies that $\langle(a,b),(c,d)\rangle$ is nilpotent. For the converse, note that there must be an n such that every product of n elements of $\langle(a,b),(c,d)\rangle$ must be zero. But every product of n elements of $\langle a,c\rangle$ occurs as the first coordinate of such a product, and so is zero. Similarly, for $\langle b,d\rangle$.
\end{proof}




\begin{corollary}
Let $L_1$ and $L_2$ be Lie algebras over a field $F$. Then $\nil(L_1\oplus L_2) = \nil(L_1) \oplus \nil(L_2)$.
\end{corollary}
 
\begin{proof}
\begin{align*}
(g, h) \in \nil(L_1) \oplus \nil(L_2) & \Leftrightarrow \langle g, x \rangle, \langle h, y \rangle \in \mathfrak{N} \ \ \forall x \in L_1 \ \ \forall y \in L_2\\
& \Leftrightarrow \langle (g, h), (x, y) \rangle \ \in\mathfrak{N} \ \  \forall  (x, y) \in L_1\oplus L_2\\
& \Leftrightarrow (g, h) \in \nil(L_1\oplus L_2).
\end{align*}  
\end{proof}

In general, If $L$ decomposes as a direct sum of ideals $L = L_1\oplus L_2 \oplus \cdots \oplus L_k$, then one can show:
\[\nil(L) = \nil(L_1) \oplus \nil(L_2) \oplus \cdots \oplus \nil(L_k).\]
That is, the nilpotentizer of a direct-sum algebra is the direct sum of the nilpotentizers of its summands.





\begin{proposition}
Let $L_1$ and $L_2$ be non-nilpotent finite-dimensional Lie algebra over a field $F$. Then $\Gamma_{\mathfrak{N}}(L_1\oplus L_2)$ is connected.
\end{proposition}

\begin{proof}
Let $(x_1, y_1), \, (x_2, y_2) \in L_1\oplus L_2 \setminus \nil(L_1\oplus L_2)$. Then $x_1 \notin \nil(L_1)$ or $y_1 \notin \nil(L_2)$ and $x_2 \notin \nil(L_1)$ or $y_2 \notin \nil(L_2)$. Consider the following cases:

\noindent \textbf{Case 1.} $x_1 \notin \nil(L_1)$ and $y_2 \notin \nil(L_2)$. In this case the path $(x_1, y_1)$, $(x_1,0)$, $(0, y_2)$, $(x_2, y_2)$ connecting the vertices $(x_1, y_1), (x_2, y_2)$. 

\noindent \textbf{Case 2.} $x_1 \in \nil(L_1)$ and $x_2 \in \nil(L_1)$. Since $L_1$ is non-nilpotent, there exists $z \in L_1\setminus  \nil(L_1)$. Then $(x_1, y_1)$, $(z, 0)$, $(x_2, y_2)$ is a path connecting   $(x_1, y_1)$ and $(x_2, y_2)$. 

The remaining cases are analogous to one of the above.
\end{proof}

\par

\begin{theorem}
\label{KappaProductNilpotent}
Let $L_1$ be a finite dimensional non-nilpotent Lie algebra. For every finite dimensional nilpotent Lie algebra $L$, it holds $ \kappa(L_1\oplus L) = \kappa(L_1)$. In particular, $\Gamma_{\mathfrak{N}}(L_1)$ is connected if and only if $\Gamma_{\mathfrak{N}}(L_1\oplus L)$ is connected for every finite dimenional nilpotent  Lie algebra $L$.
\end{theorem}

\begin{proof}
Let $L$ be any finite dimensional nilpotent Lie algebra. Then $\nil(L) = L$ and $\nil(L_1\oplus L) = \nil(L_1)\oplus L$. Thus, by Lemma \ref{ComponentsNilAdjacency},  
$(z,w), (x,y)$ are adjacent in $\Gamma_{\mathfrak{N}}(L_1\oplus L)$ if and only if $z$ and $x$ are adjacent in $\Gamma_{\mathfrak{N}}(L_1)$. Then, there is no path from $z$ to $x$ if and only if there is no path from $(z,w)$ to $(x,y)$, for all $w, y \in L$. This implies that $\kappa(L_1\oplus L) = \kappa(L_1)$.
\end{proof}

\par

\section{Strongly self-centralising subalgebras}

\begin{definition}
A subalgebra $U$ of $L$ is said to be strongly self-centralizing if $C_L(x) = U$ for all $0\neq x\in U$.
\end{definition}

\begin{example}
Let $L$ be the two-dimensional non-abelian Lie algebra. The $L$ has basis $\{x, y\}$ with $[x,y]=-[y,x]=y$. Put $U = Fx$. Then $U$ is a strongly self-centralising Lie subalgebra of $L$. 
\end{example}

\begin{example}
Let $L$ be three-dimensional with basis $\{e, f, g\}$ and products $[e,f] = 0, [e,g] = e, [f,g] = f$. Then $U := Fe + Ff$ is strongly self-centralising Lie subalgebra of $L$. 
\end{example}

\begin{example}
Let $L=\mathfrak{sl}_n(F_q)$ be the Lie algebra of $n\times n$ traceless matrices over the finite field $F_q$, and let $U$ be the subalgebra of all diagonal matrices in $\mathfrak{sl}_n(F_q)$, with $q= p^a$ and $p\nmid n$. Then $C_L(x)=U$ for all $0\neq x\in U$. 
\end{example}

\par

\begin{lemma}\label{omega}
 Let $U$ be a nilpotent subalgebra of $L$ such that $U \not \subseteq \nil(L)$. Then
 $U\setminus (U\cap \nil(L))$ is the set of vertices of a clique of $\Gamma_\mathfrak{N}(L)$. Moreover, if $N$ is a nilpotent ideal of $L$ and $L$ is finite, $|N\setminus (N\cap \nil(L))| \leq \omega(\Gamma_\mathfrak{N} (L))$. In
 particular, $|N(L) \setminus (N(L)\cap \nil(L))|\leq \omega(\Gamma_\mathfrak{N} (L))$.
\end{lemma}

\begin{proof}
Since any subalgebra of a nilpotent Lie algebra is nilpotent, $\langle x,y\rangle$ is nilpotent for all $x, y \in U \setminus (U \cap \nil(L))$. Thus the subgraph of $\Gamma_\mathfrak{N} (L)$ having $U \setminus (U\cap \nil(L))$ as its set of vertices is a clique. Moreover, if $N$ is a nilpotent ideal of $L$, then
 \begin{align*}
 |N\setminus(N\cap \nil(L))| &\leq \max\{|U \setminus(U\cap \nil(L))| : U \hbox{ is a nilpotent ideal of } L\} \\
 &\leq \max\{|U \setminus(U\cap \nil(L))| : U \hbox{ is a nilpotent subalgebra of } L\} \\
 &\leq \omega(\Gamma_\mathfrak{N}(L)).
\end{align*}

The rest follows from the fact that the nilradical $N(L)$ is the largest  nilpotent subgroup of $L$.  
\end{proof}

\par

\begin{theorem}
\label{ProperCentralizer2}
 If there exists a strongly self-centralizing subalgebra $U$ of $L$, then:
 \begin{enumerate}
     \item $\nil_L(x) = U$ for all $0\neq x\in U$.
     
     \item $\Gamma_\mathfrak{N} (L)$ is a disconnected graph.

    \item  if $L$ is finite, $|\mathrm{N}(L)|-1\leq \omega(\Gamma_\mathfrak{N} (L))$.
 \end{enumerate}
\end{theorem}
\begin{proof}
(1) Let $0\neq x\in U$. Clearly, $U=C_L(x)\subseteq \nil_L(x)$. So let $y\in \nil_L(x)$. Then there exists $z\in Z(\langle x,y\rangle)$ and so $z\in C_L(x)=U$. It follows that $y\in C_L(z)=U$, whence $\nil_L(x)\subseteq U$ and equality results.

(2) As in (1), for every $y\in L\setminus U$, $y\notin \nil_L(x)$, so $\langle x,y\rangle$ is not nilpotent and there is no edge from $y$ to $x$.

(3) By part (1) and Theorem \ref{z*},  $\nil(L) = Z^*(L) = 0$. Therefore $|N(L)|-1 \leq \omega(\Gamma_\mathfrak{N} (L))$ (by Lemma \ref{omega}).
\end{proof}


Finally, we present the complement of the nilpotent graph and an example of it.  This complement condition immediately implies that such a graph is connected. 

\begin{definition}
Let $L$ be a finite-dimensional non-nilpotent Lie algebra. The non-nilpotent graph of $L$, denoted by $\Gamma_{\mathfrak{N}}(L)^c$, is a simple undirected graph whose vertex set is $L\setminus \nil(L)$, and two vertices $x$ and $y$ are adjacent if and only if $\langle x, y\rangle$ is a non-nilpotent subalgebra of $L$.    
\end{definition}

\begin{example}
\begin{enumerate}
    \item The Non-nilpotent graph of $\mathfrak{t}(2,\mathbb{F}_2)$ consists of one connected component
    \begin{figure}[H]
    \centering
    \begin{tikzpicture}
        \GraphInit[vstyle=Shade]
        \tikzset{VertexStyle/.append style={minimum size=1pt, inner sep=1pt}}
        \renewcommand*{\EdgeLineWidth}{0.5pt}
        \Vertex[x=1, y=0]{1}
        \Vertex[x=1 , y=2]{4}    
        \Vertex[x=2.5 , y=0]{6}  
        \Vertex[x=3 , y=1]{3} 
        \Vertex[x=2.5 , y=2]{2}   
        \Vertex[x=0.5 , y=1]{5}  
        \Edges(1,6)
        \Edges(1,5)
        \Edges(1,2)
        \Edges(6,3)
        \Edges(6,4)
        \Edges(5,4)
        \Edges(5,6)
        \Edges(5,2)
        \Edges(4,2)
        \Edges(3,4)
        \Edges(3,2)
        \Edges(3,1)
    \end{tikzpicture}
    \caption{Non-nilpotent Graph of $\mathfrak{t}(2,\mathbb{F}_2)$}
    \end{figure}
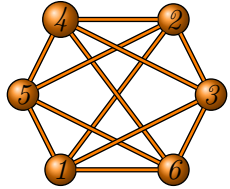
\end{enumerate}
\end{example}

\quad

\noindent \textbf{Questions.} 
\begin{itemize}
\item Is $\nil(L)$ always a subspace?
\item Does Theorem 3.4 (ii) hold over a field of prime characteristic?
\item Is the nilpotent graph of a finite-dimensional Lie algebra $L$ over $F_q$ bipartite if and only if $L$  is isomorphic to $\mathfrak{t}(2,\mathbb{F}_2)$?
\item What can we say about the classical invariants of the non-nilpotent group of a Lie algebra?
\end{itemize}

\appendix

\section{Algorithms}

\begin{algorithm}[H]
\caption{An algorithm to generate the nilpotent and non-nilpotent graphs of a finite  matrix Lie algebra $L$}
    \begin{algorithmic}
        \State $n \gets \text{Any positive integer}$
        \State $F \gets \text{Any finite Field}$
        \State $L \gets gap(f'MatrixLieAlgebra(\{F\},\{n\})')$
        \Statex
        \State $elements \gets L.\text{AsList}()$
        \State $Nilpotentizer \gets Nilpotentizer(L)$ 
        
        \State $nodes \gets [x \in elements \mid x \notin Nilpotentizer]$

        \Statex \# nilpotent and non-nilpotent graph respectively
        \State $G \gets \text{Graph}()$
        \State $H \gets \text{Graph}()$
        \State $G.\text{add\_vertices}(nodes)$
        \Statex
        
        \For{$i \gets 0$ to $\text{length}(nodes)$}
        
            \For{$j \gets i + 1$ to $\text{length}(nodes)$}
            
                \State $S \gets L.\text{Subalgebra}([nodes[i], nodes[j]])$
                
                \If{$S.\text{IsLieNilpotent}()$}
                
                    \State $G.\text{add\_edge}(nodes[i], nodes[j])$

                \Else

                    \State $H.\text{add\_edge}(nodes[i], nodes[j])$
                
                \EndIf
            \EndFor
        \EndFor
        \Statex \# Nilpotent Graph
        \State $G.\text{show}(\text{figsize}=[10, 10])$
        \Statex \# Non-Nilpotent Graph
        \State $H.\text{show}(\text{figsize}=[10, 10])$
        
    \end{algorithmic}
\end{algorithm}

\quad 

\bibliographystyle{plane}

\end{document}